\documentclass[a4paper]{scrartcl}

\usepackage{fullpage}

\usepackage{lmodern}
\usepackage[usenames,dvipsnames]{xcolor}

\usepackage{amsmath}%
\usepackage{amsfonts}%
\usepackage{amssymb}%
\usepackage{esint}
\usepackage{enumitem}
\usepackage{appendix}

\usepackage{mathtools}

\usepackage{authblk}

\usepackage[french,ngerman,british]{babel}

\usepackage[stable]{footmisc}
\usepackage{footnote}
\makesavenoteenv{align*}

\usepackage{amsthm}

\usepackage{aliascnt}

\newtheorem{theorem}{Theorem}
\theoremstyle{plain}

\newaliascnt{corollary}{theorem}  
\newtheorem{corollary}[corollary]{Corollary}
\aliascntresetthe{corollary}

\newaliascnt{lemma}{theorem}  
\newtheorem{lemma}[lemma]{Lemma}  
\aliascntresetthe{lemma}

\newaliascnt{proposition}{theorem}  
\newtheorem{proposition}[proposition]{Proposition}  
\aliascntresetthe{proposition}

\theoremstyle{definition}

\theoremstyle{remark}
\newtheorem*{note}{Note}
\newtheorem{remark}[theorem]{Remark}
\numberwithin{equation}{section}

\DeclareMathOperator*{\esssup}{ess\,sup}
\DeclareMathOperator*{\essinf}{ess\,inf}
\newcommand{\abs}[1]{\left|#1\right|}
\newcommand{\norm}[1]{\left\lVert#1\right\rVert}
\newcommand{\reals}{\ensuremath{\mathbb{R}}}
\newcommand{\complexes}{\ensuremath{\mathbb{C}}}
\newcommand{\tow}{\rightharpoonup}

\usepackage{filecontents}

\usepackage{hyperref}

\renewcommand{\leq}{\leqslant}
\renewcommand{\geq}{\geqslant}

\title{Bound States of the One-Dimensional Maxwell--Schr\"odinger Equations}
\author{Richard Chapling}

\affil{Department of Applied Mathematics and Theoretical Physics, \\ University of Cambridge, Cambridge, England}

\begin{document}

\maketitle

\begin{abstract}
	We prove the existence of a ground state of the Maxwell--Schr\"odinger equations in one spatial dimension, describing a specified amount of free charge under the influence of a fixed charge. For one case (equal free and fixed charge, i.e., a neutral atom), we introduce a new type of quartic Banach space, in which the Hamiltonian is naturally coercive. We also show that for a point charge, for any ratio of charge such that the ground state exists, it is symmetric and decreasing.
\end{abstract}



\section{Introduction}

The Maxwell--Schr\"odinger equations in one dimension, in the presence of a fixed background charge \(\rho\), can be shown to reduce to
\begin{equation}
\label{eq:s-c}
\begin{aligned}
	-u''+Vu &=\epsilon u, \\
	-V'' &=u^{2}+\rho
\end{aligned}
\end{equation}
(we shall demonstrate this in \S~\ref{sec:probcons}). We shall assume that \(u\) describes a fixed amount of charge,
\begin{equation}
	\int_{\reals} u^{2} = 1.
\end{equation}
Our analysis mainly focuses on a variant of the functional which gives \eqref{eq:s-c} as its Euler--Lagrange equations, namely
\begin{equation}
	E[u] := \int_{\reals} \abs{u'(x)}^{2} \, dx + \frac{1}{2} \int_{\reals} \int_{\reals} -\abs{x-y} [(u(x))^{2}+\rho(x)] [(u(y))^{2}+\rho(y)] \, dx \, dy.
\end{equation}

Our initial discussion revolves around the particular \( \rho \) corresponding to a point charge at the origin, i.e. \( \rho = -z\delta_{0} \), so a solution to the equations \eqref{eq:s-c} will describe the probability density of a bound state of electrons in a one-dimensional atom. This shall then be used to extend the results to any negative potential satisfying \( \int_{\reals} \abs{x} \abs{\rho(x)} \, dx < \infty \).

It is well-known that describing the electron field of an atom using the Schr\"odinger equation is inadequate even at low energy, since the self-interaction of the field is not present; this has significant implications in areas of quantum chemistry and semiconductor physics, for example. An initial way of resolving this theoretically is to write down the full Lagrangian of the Schr\"odinger field describing the electrons interacting with a Maxwellian electromagnetic field; the appropriate Euler--Lagrange equations then describe the  system in a consistent way.\footnote{For this system's application in semiconductor physics, see e.g. \cite{rossi2011theory}, and in quantum plasmai, \cite{Manfredi:1994kx}.}\textsuperscript{,}\footnote{c.f. the approach of the Hartree equation \cite{Hartree:1928uq,Lions19811245}, which has similar, but less general derivation.} The counterbalance of this benefit is the loss of simple analytic solutions, and so it is necessary to investigate the analytic properties of this system to discover if it is a sensible one to use as a model.

Unlike the Schr\"odinger--Newton equations, there does not exist a bound state when the charged field has no binding background opposite charge: this agrees with our natural intuition that as like charges repel, the charge distribution will naturally disperse to infinity. Hence it is natural to couple the fields to a fixed background charge distribution; in the example given above this corresponds to the nucleus of the atom, normally modelled as a point charge.

In this paper we shall discuss the one-dimensional analogue of the familiar three-dimensional version of this problem: this has advantage of simplicity, but requires quite different analytic equipment from higher-dimensional cases.

\subsection{Prior work}

The three-dimensional case was discussed by Coclite and Georgiev \cite{Coclite03solitarywaves}, who proved the existence of bound states of fixed norm and some of their properties, as well as a nonexistence result for the ``negative ion'' case of more moving charge than fixed. Existence of a bound state of the equations on a compact subset of \( \reals^{3} \) with specified electrical potential on the boundary was proven by Benci and Fortunato \cite{fortunato1998eigenvalue}.

The corresponding examples for the Schr\"odinger--Newton equations are well-understood: Lieb \cite{Lieb:1977uq} proved existence and uniqueness of the minimising solution in three spatial dimensions, and Choquard and Stubbe \cite{Choquard:2007fk} proved existence and uniqueness of the ground state in one spatial dimension. We adapt some of their methods, but the ``neutral atom'' case of equal fixed charge and moveable charge requires a more subtle argument, for which we introduce a new type of Banach space, which extends the idea of a weighted Sobolev space and the Rellich criterion.

\subsection{Results}
\newcommand{\cX}{\mathcal{X}}
\newcommand{\cB}{\mathcal{B}}
\newcommand{\fB}{\mathfrak{B}}

We prove (\autoref{thm:z>1exists}) that the functional \(E\) has a minimiser if \(z > 1\) (a ``positive ion''): for \( z>1 \), it acts in the far field, as one might expect, like a particle in the well of a point particle of charge \( z-1 \); this causes it to have sufficiently rapid decay that the first absolute moment of charge about the origin is finite, \textit{viz.} \( \int \abs{x} u(x)^{2} \, dx < \infty \).

For \( z=1 \) (a ``neutral atom''), the situation is more complicated, as is reflected in the greater complexity of the space we work in: the obvious Hilbert space \( \cX \) is replaced by the quartic Banach space \( \cB \).\footnote{See \S~2 for the definitions of these spaces.} However, this new space does provide us with coercivity, allowing us to demonstrate (\autoref{thm:z=1exists}) that the energy still achieves a minimum.

For \( z<1 \), we demonstrate that in this system, negative ions cannot exist: in \S~\ref{sec:z<1counterexample} we provide an example that shows that the energy is unbounded below, so it has no minimum.

\subsection{Derivation of the Equations Considered}
\label{sec:probcons}
Consider the Maxwell--Schr\"odinger Lagrangian density in $(1+1)$ dimensions:
\begin{align}
\label{eq:m--s}
	\mathcal{L} &= -\frac{1}{4} F_{\mu\nu} F^{\mu\nu} - A_{\mu} J^{\mu} + i \overline{\psi} \partial_{0} \psi - \frac{1}{2m} \nabla \overline{\psi} \cdot \nabla \psi \\
	&= \frac{1}{2} (\partial_{0}A_{1}-\partial_{1} A_{0})^{2} - A_{0}J^{0} + A_{1}J^{1} + i \overline{\psi} \partial_{0} \psi - \frac{1}{2m} \partial_{1} \overline{\psi} \partial_{1} \psi,
\end{align}
where \(J^{0} = q(j^{0}+\abs{\psi}^{2})\) and \(J^{1}=q\Im(\overline{\psi} \partial_{1} \psi )\), so we have a background charge and no background current. The Euler--Lagrange equations for \( A \) are
\begin{align*}
	\partial_{0} (\partial_{0} A_{1} - \partial_{1} A_{0}) &= J^{1}, \\
	-\partial_{1} (\partial_{1} A_{0} - \partial_{0} A_{1}) &= J^{0}
\end{align*}
In one dimension, we can force \(A_{1}=0\) using a gauge transformation, which reduces the Lagrangian to
\begin{equation*}
	\mathcal{L} = \frac{1}{2} (\partial_{1} A_{0})^{2} - A_{0}J^{0} + i \overline{\psi} \partial_{0} \psi - \frac{1}{2m} \partial_{1} \overline{\psi} \partial_{1} \psi.
\end{equation*}
We shall look for stationary states, \(\psi(t,x) = e^{-i\lambda t} \psi(x)\), \(\psi(x)\) real-valued, \(\partial_{0}A_{0} = 0\), leaving us with a final Lagrangian
\begin{equation*}
	\mathcal{L} = - \frac{1}{2m} \abs{\psi'}^{2} + \frac{1}{2}\Phi'^{2} -q\Phi(\varrho+\abs{\psi}^{2}) +\lambda \abs{\psi}^{2},
\end{equation*}
where we have written $\Phi=A_{0}$ and $\varrho=j^{0}$ for brevity.

Thus in this case the equations reduce to the Schr\"odinger--Coulomb equations,
\begin{align}
	-\psi'' + q \Phi \psi &= \lambda \psi \\
	-\Phi'' &= q(\psi^{2}+\varrho),
\end{align}
where \(\int \varrho = -Z \), and subject to \( \int \abs{\psi}^{2} = N \). Rescaling with \( u = N^{-1/2}\psi \), \( \rho = N^{-1} \varrho \) and \( V = N^{-1/2} \Phi \), the equations become
\begin{align}
	-u'' + N^{1/2}qVu &= \lambda u \\
	-V'' &= N^{1/2} q(u^{2}+\rho),
\end{align}
and then if we rescale the variable to \( y = N^{1/4}q^{1/2}x \), we finally attain the equations promised above,
\begin{align}
	-u'' + Vu &= \epsilon u \\
	-V'' &= u^{2}+\rho,
\end{align}
where \( N^{1/2} q \epsilon = \lambda \), and \( \int \rho = -Z/N =: -z \), subject to \( \int u^{2} = 1 \). This has non-dimensionalised the equations and tells us that the only significant number for examining the system is the charge ratio \(z\).

The obvious functional that gives these equations is
\begin{equation}
\label{eq:badfunctional}
	\int_{\reals} \left( u'^{2} + V(u^{2} + \rho) - \frac{1}{2} V'^{2} \right);
\end{equation}
however, this is obviously not \emph{a priori} bounded below; since \(V\) is the problem, it is better to solve the Coulomb equation and produce a different functional, which in turn produces one non-local interaction equation. It is easy to see that
\begin{equation}
	V(x) = -\frac{1}{2} \int_{\reals} \abs{x-y} (u(y)^{2}+\rho(y)) \, dy
\end{equation}
solves the Coulomb equation, so we may take the functional instead to be
\begin{equation}
\label{eq:energy}
	E[u] := \int_{\reals} \abs{u'(x)}^{2} \, dx + \frac{1}{2} \int_{\reals} \int_{\reals} -\abs{x-y} [(u(x))^{2}+\rho(x)] [(u(y))^{2}+\rho(y)] \, dx \, dy.
\end{equation}

\begin{remark}
	This functional is exactly equivalent to that given in \eqref{eq:badfunctional} if \( [VV']_{-\infty}^{\infty} =0 \), which we can see by integrating by parts; we find by consideration of the explicit forms of \(V\) and \(V'\) that sufficient conditions for this are \( z=1 \) and \( \int_{\reals} \abs{y}(u(y)^{2}+\rho(y)) \, dy  \) finite. Hereafter we consider \( E[u] \) the essential quantity, and investigate it under more general conditions.
\end{remark}

We shall initially take the singular background 
\begin{equation}
\label{eq:rhodef}
	\rho(x)=-z\delta(x),
\end{equation}
which may be thought of as a point particle with charge \(z\) at the origin.\footnote{And hence it is sensible to refer to this as a ``one-dimensional hydrogen-like atom''; it will become apparent that the one-dimensional case is sufficiently regular to allow us to work directly with this background distribution.} It generates the potential
\begin{equation*}
	V_{\rho}(x) = \frac{z}{2}\abs{x};
\end{equation*}
it is easy to check that this satisfies \( -V_{\rho}'' = \rho \) in the sense of distributions.

\section{Definitions and theory}

\subsection{Definitions}

We record here for definiteness the spaces considered in this paper.

We define first the functionals
\begin{align}
	b[f,g] &= \iint_{X \times X} f(x) g(y) d\mu(x,y) \\
	B[f] &= b[f,f],
\end{align}
where \( d\mu(x,y) = G(x,y) \, dx \times dy \), where \( G(x,y) = G(y,x) \) is symmetric, and positive-semidefinite in the sense that
\begin{equation*}
	b[f,f] \geq 0
\end{equation*}
for all \( f \in L^{1}(X) \). (In this paper, the measure space \(X\) is simply \(\reals\), but, as shall be noted below, some of the results are more generally true.)

Define the following spaces:
\begin{align}
\label{eq:spacedefns}
	H^{1} &= \left\{ (u: \reals \to \reals) : (\norm{u}_{H}^{2} :=) \norm{u'}_{2}^{2} + \norm{u}_{2}^{2} < \infty \right\}, \\
	\cX &= \left\{ u \in H^{1}: \int_{\reals} \abs{x} u(x)^{2} \, dx < \infty \right\} \\
	\cB &= \left\{ u \in H^{1}: B[u^{2}] < \infty \right\} \\
	\fB &= \left\{ f \in L^{1}(\reals): B[\abs{f}] < \infty \right\}, \\
	\intertext{and finally, if \(A\) is one of the above spaces, we write}
	A_{1} &= \{ u \in A : \norm{u}_{2} = 1 \}.
\end{align}

\subsection{The space \texorpdfstring{\( \cX \)}{cX}}
We have the following theorem:\footnote{Reed and Simon, \cite{reed1978analysis} XIII.65, also known as the ``Rellich Criterion''}
\begin{theorem}
\label{thm:XHilbert}
	\( \cX \) is a Hilbert space with inner product
	\begin{equation}
		\langle u, v \rangle_{\cX} = \int_{\reals} \left( u' v' + (1+\abs{x})uv \right) \, dx,
	\end{equation}
	and \( \cX \subset L^{2} \), and any weakly convergent sequence in \(\cX\) has a subsequence that is strongly convergent in the \( L^{2} \)-norm (i.e. \( \cX \) is compactly contained in \( L^{2} \), written \( \cX \Subset L^{2}(\reals)\)).
\end{theorem}

\subsection{The \texorpdfstring{\( B \)}{B} spaces}
\label{sec:Bspaces}

Let \( \norm{u}_{B} = B[u^{2}]^{1/4} \). This section is devoted to proving
\begin{theorem}
\label{thm:BBanach}
	If \( B \) is positive-semidefinite, \( \cB \) with the norm
	\begin{equation}
		\norm{u}_{H^{1}} + \norm{u}_{B}
	\end{equation}
	is a Banach space.
\end{theorem}

\begin{lemma}
	If \( B \) is positive-(semi)definite, \( b[f,g] \) is a (degenerate) inner product on \( \fB \). If \( B \) is nondegenerate, it turns \( \fB \) into a Hilbert space.
\end{lemma}
\begin{proof}
	Bilinearity is obvious, as is symmetry. Positive-definiteness (or non-negative-definiteness) follows from the property of \(B\) assumed.
	
	In the non-degenerate case, we obtain the Hilbert space structure as follows: since \( L^{1} \) is complete, it suffices to check that \( B[\abs{\cdot}] \) being finite is a closed condition with respect to the norm \( B[\cdot]^{1/2} \). This is straightforward, and essentially works in the same way as the proof of \autoref{thm:BBanach} below. 
\end{proof}

\begin{lemma}
	If \( B \) is positive-(semi)definite, \( \norm{\cdot}_{B}\) is a (semi)norm on \( \cB \).
\end{lemma}

\begin{proof}
	Absolute homogeneity,
	\begin{equation*}
	\norm{\lambda u}_{B} = \abs{\lambda} \norm{u}_{B} \qquad \forall \lambda \in \reals, u \in \cB,
	\end{equation*}
	is obvious. For the triangle inequality, notice that
	\begin{equation*}
		d\nu(x) := \int_{y} (u+v)^{2}(y) d\mu(x,y)
	\end{equation*}
	is a measure on \( \cB \), so we have a triangle inequality
	\begin{equation*}
		B[(u+v)^{2}]^{1/2}=\left(\int (u+v)^{2}(x) d\nu(x)\right)^{1/2} \leq \left(\int u^{2}(x) d\nu(x)\right)^{1/2} + \left(\int v^{2}(x) d\nu(x)\right)^{1/2}
	\end{equation*}
	If we now interchange the order of integration, we can carry out the same procedure again, using \( d\nu_{u}(y) = \int_{x} u^{2}(x) d\mu(x,y) \):
	\begin{align*}
		\left(\int u^{2}(x) d\nu(x)\right)^{1/2} &= \left(\int_{y} (u+v)^{2}(y) \int_{x} u^{2}(x) d\mu(x,y)\right)^{1/2} \\
		&= \left( \int_{y} (u+v)^{2}(y) d\nu_{u}(y) \right)^{1/2} \\
		&\leq \left( \int_{y} u^{2}(y) d\nu_{u}(y) \right)^{1/2} + \left( \int_{y} v^{2}(y) d\nu_{u}(y) \right)^{1/2} \\
		&= B[u^{2}]^{1/2}+b[u^{2},v^{2}]^{1/2},
	\end{align*}
	and similarly with \(v\), to obtain
	\begin{equation*}
		B[(u+v)^{2}]^{1/2} \leq B[u^{2}]^{1/2}+2b[u^{2},v^{2}]^{1/2}+B[v^{2}]^{1/2},
	\end{equation*}
	using the symmetry of \(d\mu\) on the cross terms. To eliminate this cross term (which is obviously positive), we can use the Cauchy--Schwarz inequality:
	\begin{equation*}
		b[u^{2},v^{2}] \leq B[u^{2}]^{1/2}B[v^{2}]^{1/2},
	\end{equation*}
	which leads us to
	\begin{align*}
		\norm{u+v}_{B} = (B[(u+v)^{2}]^{1/2})^{1/2} &\leq \left( B[u^{2}]^{1/2} + 2 b[u^{2},v^{2}] + B[v^{2}]^{1/2} \right)^{1/2} \\
		&\leq \left( B[u^{2}]^{1/2} + 2 B[u^{2}]^{1/4}B[v^{2}]^{1/4} + B[v^{2}]^{1/2} \right)^{1/2} \\
		&= \left( ( B[u^{2}]^{1/4} + B[v^{2}]^{1/4} )^{2} \right)^{1/2} \\
		&= B[u^{2}]^{1/4} + B[v^{2}]^{1/4} = \norm{u}_{B} + \norm{v}_{B},
	\end{align*}
	as required.
\end{proof}

\begin{proof}[Proof of Theorem]
	This is now simple: since both norms are norms on \( \cB \), so is their sum. Completeness is simple: \( H^{1} \) is complete, so it suffices to check that if \( u_{m} \to u \) is a Cauchy sequence for \( \norm{\cdot}_{H^{1}} + \norm{\cdot}_{B} \), then \( B[u^{2}] < \infty \). But this is clear: the numbers \( \norm{u_{m}}_{B} = B[u_{m}^{2}]^{1/4} \) form a real-valued Cauchy sequence, so they have a finite limit, which by strong convergence is \( B[u^{2}]^{1/4} \).
\end{proof}

We shall also need to produce weakly convergent subsequences in \(\cB\); the compactness that enables this is a consequence of
\begin{lemma}
\label{thm:Bunifconvex}
	\( \norm{\cdot}_{B} \) is uniformly convex on \(L^{2B}(X) := \{ u \in L^{2}(X): B[u^{2}]<\infty \}\).
\end{lemma}

\begin{proof}
	By expansion, we have
	\begin{equation*}
		B[(u+v)^{2}]+B[(u-v)^{2}] = 2(B[u^{2}]+B[v^{2}]+2b[u^{2},v^{2}])+2B[2uv],
	\end{equation*}
	and \( 2uv < u^{2}+v^{2} \), so
	\begin{equation*}
		B[2uv] \leq B[u^{2}+v^{2}] = B[u^{2}]+B[v^{2}]+2b[u^{2},v^{2}] \leq (B[u^{2}]^{1/2}+B[v^{2}]^{1/2})^{2},
	\end{equation*}
	by Cauchy--Schwarz or the triangle inequality, and we then have
	\begin{equation*}
		B[(u+v)^{2}]+B[(u-v)^{2}] \leq 4(B[u^{2}]+B[v^{2}]+2b[u^{2},v^{2}]) \leq 4(B[u^{2}]^{1/2}+B[v^{2}]^{1/2})^{2},
	\end{equation*}
	or
	\begin{equation*}
		\norm{u+v}_{B}^{4}+\norm{u-v}_{B}^{4} \leq 4(\norm{u}_{B}^{2}+\norm{v}_{B}^{2})^{2},
	\end{equation*}
	Now suppose that \( \norm{u}_{B}=\norm{v}_{B}=1 \), and \( \norm{u-v}_{B} \geq \varepsilon \). Then
	\begin{equation*}
		\norm{\tfrac{1}{2}(u+v)}_{B}^{4} \leq \frac{1}{4} (\norm{u}_{B}^{2}+\norm{v}_{B}^{2})^{2} - \frac{1}{16}\norm{u-v}_{B}^{4} \leq 1-\delta,
	\end{equation*}
	for some \( \delta=\delta(\varepsilon)>0 \), which is exactly uniform convexity.
\end{proof}

\begin{corollary}
\label{thm:LBreflex}
	\( L^{2B}(X) \) equipped with the \(B\)-norm is reflexive.
\end{corollary}

\begin{lemma}
	\(\cB\) with norm \( \norm{u}_{H}+\norm{u}_{B} \) is reflexive.
\end{lemma}

This is a consequence of \cite{Liu:1969ly}, Cor.~1.8.; we recall that the maximum of two norms is equivalent to their sum. We can also give a direct proof:

\begin{proof}
	By results of Bourbaki, Kakutani, Shmulyan and Eberlein,\footnote{\cite{Pietsch:2007zr}, p.~80, 3.4.3.6f.} it suffices to check that each bounded sequence has a weakly convergent subsequence. Boundedness implies that \( \norm{u_{m}}_{H} \) and \( \norm{u_{m}}_{B} \) are both bounded separately. Since \( H^{1} \) is reflexive, the Banach--Alaoglu theorem implies it is weakly compact, so there is a subsequence \( (u_{m'}) \subseteq (u_{m}) \) that converges weakly in \( H^{1} \). But by our assumption, this subsequence is also bounded in \( L^{2B} \), which uniformly convex and so reflexive, and we can apply Banach--Alaoglu to extract another subsequence \( (u_{m''}) \subseteq (u_{m'}) \) which is weakly convergent in \( L^{2B} \). But \( (u_{m''}) \) is also weakly convergent in \( H^{1} \), so it is weakly convergent in both spaces and hence in their intersection \(\cB\).
\end{proof}

\begin{lemma}
\label{thm:Cposdef}
	\begin{equation}
		C[f] = \iint_{xy>0} f(x) f(y) \min{\{ \abs{x}, \abs{y} \}} \, dx \, dy
	\end{equation}
	is positive-definite.
\end{lemma}
This is clear after considering the alternative form of \( C \),
\begin{equation}
	\int_{0}^{\infty} \left( \int_{z}^{\infty} f(x) \, dx \right)^{2} \, dz,
\end{equation}
which we derive in \S~\ref{sec:formspoten}.

\section{The one-dimensional hydrogen-like atom}

With the preliminary constructions complete,\footnote{Pun not intended} we shall examine the functional \eqref{eq:energy} in detail, and derive a more useful expression for the Coulomb term.

Inserting the point charge distribution \autoref{eq:rhodef}, the energy functional becomes
\begin{align}
	E[u] &= \int_{\reals} \abs{u'(x)}^{2} \, dx + \frac{1}{2} \int_{\reals} \int_{\reals} -\abs{x-y} (u(x))^{2} (u(y))^{2} \, dx \, dy + z \int_{\reals} \abs{x} (u(x))^{2} \, dx \\
	&=: \norm{u'}_{2}^{2} + C[u^{2}].
\end{align}

\begin{note}
	The \( \rho\rho \) term is fortunately absent, as we might expect: we may see this by taking a sequence of approximants \( \delta_{n}(x) := n \phi(nx) \), where \(\phi\) is a smooth, compactly supported function with \( \int \phi = 1 \). Then \( \delta_{n} \to \delta \) as a distribution, but one can apply a scaling argument and the triangle inequality to show that
	\begin{equation}
		\iint \abs{x-y} \delta_{n}(x) \delta_{n}(y) \, dx \, dy = O\left( \frac{1}{n} \right) \quad \text{as }n \to \infty.
	\end{equation}
\end{note}

At this point we have a number of tricks to apply: the single integral in \(C\) can be expressed as a symmetric double integral by recalling that \( \norm{u}_{2} = 1 \) and exploiting the symmetry inherent in the new integrand:
\begin{align*}
	2\int_{\reals} \abs{x} (u(x))^{2} \, dx &= 2 \left( \int_{\reals} (u(y))^{2} \, dy \right) \left( \int_{\reals} \abs{x} (u(x))^{2} \, dx \right) \\
	&= 2 \int_{\reals} \int_{\reals} \abs{x} (u(x))^{2} (u(y))^{2} \, dx \, dy \\
	&= \int_{\reals} \int_{\reals} (\abs{x}+\abs{y}) (u(x))^{2} (u(y))^{2} \, dx \, dy.
\end{align*}

Now, given this expression, we have
\begin{equation}
	C[u^{2}] = \frac{1}{2} \int_{\reals} \int_{\reals} [z(\abs{x}+\abs{y})-\abs{x-y}] (u(x))^{2} (u(y))^{2} \, dx \, dy;
\end{equation}
since we are in one dimension, the term in square brackets can be written as
\begin{equation}
	2g(x,y) := z(\abs{x}+\abs{y})-\abs{x-y} = (z-1)(\abs{x}+\abs{y}) + 2G(x,y),
\end{equation}
where
\begin{equation}
	G(x,y) := \frac{1}{2} \left( \abs{x}+\abs{y}-\abs{x-y} \right) = \begin{cases}
		\min{\{\abs{x},\abs{y}\}} & xy>0 \\
		0 & xy < 0
	\end{cases};
\end{equation}
there are now three cases:

\subsection{\texorpdfstring{\(z<1\)}{z<1}}
\label{sec:z<1counterexample}
We shall show that \(E\) is not bounded below on \( \cX \). Consider, for example, the functions
\begin{equation*}
	(u_{n}(x))^{2} = \frac{(1+n)^{3}}{2n^{3}} \left( \frac{1}{(1+\abs{x})^{2}}-\frac{1}{(1+n)^{2}} +\frac{2(\abs{x}-n)}{(1+n)^{3}}  \right) \chi_{[-n,n]}(x)
\end{equation*}
It is easy to check that these functions are all in \( \cX \):\footnote{Recall from \eqref{eq:spacedefns} that \(\cX\) is the subset of \( H^{1} \) with \( \int_{\reals} \abs{x} u(x)^{2} \, dx < \infty \).} the latter terms ensure that \(u_{n}^{2}( \pm n + \varepsilon ) = O(  \varepsilon^{2} ) \), and hence \(u'_{n}\) is appropriately square-integrable; a similar calculation verifies that \( \norm{u_{n}}_{2}=1 \). By explicit computation, we find that
\begin{equation*}
	C[u_{n}^{2}] =  (z-1) \log{(n+1)} + O(1)
\end{equation*}
as \(n\to \infty \), and in particular, if \( z<1 \), the first term can be made as negative as desired. On the other hand, the kinetic term is \( 1/6 + O(1/n) \), and hence \(E\) can be made as negative as desired. Therefore \(E\) is not bounded below on \( \cX \), and can admit no global minimiser.

\subsection{\texorpdfstring{\(z>1\)}{z>1}}
The triangle inequality shows that \( G(x,y) \geq 0 \), and therefore
\begin{equation}
	C[u^{2}] \geq \frac{1}{2}(z-1) \iint (\abs{x}+\abs{y}) (u(x))^{2} (u(y))^{2} \, dx \, dy = (z-1) \int \abs{x} (u(x))^{2} \, dx,
\end{equation}
so \( C[u^{2}] \) is coercive and positive on \( \cX \); moreover, it follows that the entire energy functional \( E \) is coercive and positive on \( \cX \).

\begin{lemma}
\label{thm:z>1wlsc}
	When \( z>1 \), \( E[u] \) is weakly lower-semicontinuous (WLSC) on \( \cX \).
\end{lemma}

\begin{proof}
	It is well-known that the kinetic term \( \norm{u'}_{2}^{2} \) is WLSC, so it suffices to check that the potential term \( B[u^{2}] \) is. Given a sequence \( (u_{m}) \subset \cX \), \( u_{m} \tow u \) in \( \cX \), we can take a subsequence \( m' \subseteq m \) so that \( u_{m'} \to u \) pointwise a.e..\footnote{See, e.g. Lieb and Loss, \cite{lieb2001analysis}, Corollary 8.7 (p.212)} Hence \( ( (x,y) \mapsto (u_{m}(x))^{2}(u_{m}(y))^{2} g(x,y) ) \) is a positive sequence of functions converging a.e. to \( ( (x,y) \mapsto (u(x))^{2}(u(y))^{2} g(x,y) ) \), and Fatou's lemma implies that
	\begin{equation}
		\liminf_{m} \iint (u_{m}(x))^{2}(u_{m}(y))^{2} g(x,y) \, dx \, dy \geq \iint (u(x))^{2}(u(y))^{2} g(x,y) \, dx \, dy,
	\end{equation}
	i.e. \( \liminf_{m} B[u_{m}^{2}] \geq B[u^{2}] \), as required.
\end{proof}

\begin{theorem}
\label{thm:z>1exists}
	If \( z>1 \), \( E \) has a minimiser in \( \cX_{1} \), i.e. if \( \liminf_{v \in \cX_{1}} E[v] = e_{0} \), \( \exists u \in \cX_{1} \) such that \( E[u] = e_{0}\).
\end{theorem}

\begin{proof}
	\( E \) is weakly lower-semicontinuous, bounded below, and coercive on \( \cX \) (see above), and by \autoref{thm:XHilbert}, \( \cX_{1} \) is a compact subset of \( L^{2}(\reals) \). It follows that
	\begin{enumerate}
		\item
		the infimum \( e_{0} \) exists and is positive,
		\item
		there is a minimising sequence \( (u_{m}) \subset \cX_{1} \)
		\item
		we can take a subsequence \( m' \subseteq m \) such that \( u_{m'} \) converges to \(u\)
		\begin{enumerate}
			\item
			weakly in \( \cX \),
			\item
			strongly in \( L^{2} \),
			\item
			pointwise almost everywhere.
		\end{enumerate}
	\end{enumerate}
	From the strong convergence it follows that \( u \in \cX_{1} \), so \( E[u] \geq e_{0} \), but by the weak lower-semicontinuity \( E[u] \leq e_{0} \). Hence \( E[u] = e_{0} \), as required.
\end{proof}

In the next section we consider the \(z=1\) case, it being the most subtle.

\section{\texorpdfstring{\(z=1\)}{z=1}}
Here, we have 
\begin{equation}
	C[u^{2}] = \frac{1}{2}\int_{\reals} \int_{\reals} [\abs{x}+\abs{y}-\abs{x-y}] (u(x))^{2} (u(y))^{2} \, dx \, dy,
\end{equation}
and now the triangle inequality implies that \(C[u^{2}] \geq 0 \), and hence \( E[u] \) is bounded below by \(0\). Further, the term in square brackets can be written as
\begin{equation}
	\abs{x}+\abs{y}-\abs{x-y} = \begin{cases}
		\min{\{\abs{x},\abs{y}\}} & xy>0 \\
		0 & xy \leq 0
	\end{cases}.
\end{equation}

We therefore have a curious decoupling result between the distribution of \(u\) on the positive and negative parts of the axis:
\begin{align}
	C[u^{2}] &= \iint_{x,y>0} \min{\{x,y\}} (u(x))^{2} (u(y))^{2} \, dx \, dy + \iint_{x,y<0} \min{\{-x,-y\}} (u(x))^{2} (u(y))^{2} \, dx \, dy \\
	&=: C_{+}[u^{2}] + C_{-}[u^{2}], \notag
\end{align}
say. It is also clear that if we write \( f_{R}(x) = f(-x) \), then \( C[f] = C_{+}[f] + C_{+}[f_{R}] \).

Since \(\min{\{\abs{x},\abs{y}\}} \leq \abs{x}\), we also have
\begin{align}
	C[u^{2}] &\leq \left( \int_{0}^{\infty} \abs{x} (u(x))^{2} \, dx \right) \left( \int_{0}^{\infty} (u(y))^{2} \, dy \right) + \left( \int_{-\infty}^{0} \abs{x} (u(x))^{2} \, dx \right) \left( \int_{-\infty}^{0} (u(y))^{2} \, dy \right) \\
	&\leq \int_{-\infty}^{\infty} \abs{x} (u(x))^{2} \, dx,
\end{align}
so certainly \(C\) is finite on \(\cX\).

However, since \( C \) may be finite when the \(\cX\)-norm is not, it is more logical to work on the larger space \(\cB\) where \(C\) is finite.

\begin{remark}
	Since \(\min{\{\abs{x},\abs{y}\}} \leq \abs{x}^{1/2}\abs{y}^{1/2} \), we also have the bound
\begin{equation*}
	C[u^{2}] \leq \left( \int \abs{x}^{1/2} (u(x))^{2} \, dx \right)^{2};
\end{equation*}
however, the subspace of \( H^{1} \) where the right hand side is finite, which we may call by analogy \( \cX^{1/2} \), is in fact a proper subspace of \( \cB \): consider the function \[ T(x) := \int_{x}^{\infty} (u(y))^{2} \, dy := \begin{cases}
		\frac{(1+x/e)^{-1/2}}{(\log{(e+x)})^{n}} & x>0 \\
		1 & \text{else}
	\end{cases}, \]
for example. With appropriate smoothing at the origin, this \(u\) can be admitted to \( H^{1}(\reals) \); if \( 1/2<n\leq 1 \), it is then in \( \cB \), but not \( \cX^{1/2} \). Hence \( \cX^{1/2} \subset \cB \), but not \emph{vice versa}.
\end{remark}

\subsection{Forms of the potential energy}
\label{sec:formspoten}

From the discussion in the previous section, it is clear that the pertinent functional to study is \( C_{+} \), since the full potential energy functional is just a sum of two different versions of \( C_{+} \). We therefore shall work with \( C_{+}[f] \) for \( f \in \fB \) (recall \eqref{eq:spacedefns}) with the norm functional \( B=C \),
\begin{equation*}
	C_{+}[f] = \int_{0<x,y<\infty} f(x) f(y) \min{\{x,y\}} \, dx \times dy.
\end{equation*}
Since by the definition of \( \fB \) \( C_{+}[\abs{f}]<\infty \), Fubini's theorem will allow us to write \( C_{+}[f] \) in a number of ways, by expressing it as an iterated integral and permuting the order of integration.

Firstly, by the symmetry of the integrand, we can always split \( C_{+} \) so that we have
\begin{equation}
	C_{+}[f] = 2\iint_{0<y<x<\infty} f(x) f(y) y \, dx \times dy.
\end{equation}
Now, we can write \( y = \int_{0}^{y} dz \), and then the integral becomes the triple integral
\begin{equation}
	\infty > C_{+}[f] = 2\iiint_{0<z<y<x<\infty} f(x) f(y) \, dx \times dy \times dz;
\end{equation}
at this point we can conclude that \( f(x) f(y) \in L^{1}((x,y,z) \in \reals^{3} : 0<z<y<x<\infty) \). Hence we can pass to an iterated integral: there are \( 3!=6 \) possible orders of integration from which to choose, but there is some symmetry, which actually reduces the number of distinct forms we obtain to four:
\begin{align*}
	&2\int_{0}^{\infty} f(x) \left( \int_{0}^{x} y f(y) \, dy \right) \, dx,	&&2\int_{0}^{\infty} y f(y) \left( \int_{y}^{\infty} f(x) \, dx \right) \, dy, \\
	&\int_{0}^{\infty} \left( \int_{z}^{\infty} f(x) \, dx \right)^{2} \, dz,	&&\int_{0}^{\infty} f(x) \left( \int_{0}^{x} \left( \int_{z}^{\infty} f(y) \, dy \right) \, dz \right) \, dx;
\end{align*}
the latter is interesting in that it does not emerge by itself, but we have
\begin{align*}
	\int_{0}^{\infty}& f(x) \left( \int_{0}^{x} \left( \int_{z}^{x} f(y) \,  dy \right) dz \right) \, dx = \int_{0}^{\infty} f(x) \left( \int_{0}^{x} \left( \int_{z}^{\infty} f(y) \, dy - \int_{x}^{\infty} f(y) \, dy \right) \, dz \right) dx \\
	&= \int_{0}^{\infty} f(x) \left( \int_{0}^{x} \left( \int_{z}^{\infty} f(y) \, dy \right) \, dz \right) dx - \int_{0}^{\infty} x f(x) \left( \int_{x}^{\infty} f(y) \, dy \right) dx,
\end{align*}
and we recognise the second term as equal to the negative of the left-hand side (by Tonelli's theorem: cf. the upper-right integral, to which it is identical). Adding it to both sides, we obtain the stated result. Indeed, this argument works even if both sides are infinite, by taking the upper limit to be a large finite constant \(K\) rather than \(\infty\), then taking the limit subsequent to the manipulations.

The most useful of these forms for our purposes is the bottom-left one: it demonstrates the positivity of \( C_{+} \) on nonzero functions (the internal integral can only be zero almost everywhere if \( f \) is zero almost everywhere, whence the entire expression can only be zero if \( f \) is zero almost everywhere, which is precisely nondegeneracy), and hence \( C_{+} \) does satisfy \autoref{thm:Cposdef}; we shall also use this later in the proof that \( \cB \Subset L^{2} \).

\subsection{Norm Convergence: A Concentration--Compactness-Type Argument}
Define the concentration functional, \( Q: L^{1}(\reals) \times [0,\infty) \to [0,1] \), by
\begin{equation}
	Q[f](r) = \int_{-r}^{r} f(x) \, dx.
\end{equation}

\begin{lemma}
	Let \( (u_{m}) \) be a sequence \( \subset L^{2}(\reals) \) with \( \norm{u_{m}}_{2} = 1 \), and write \( Q_{m} = Q[u_{m}^{2}]\). Then there is a subsequence, \( (Q_{m'}) \subseteq (Q_{m}) \), that converges almost everywhere to a nondecreasing function \( Q:[0,\infty) \to [0,1]\). This may be chosen to be continuous from the left, so that
	\begin{equation}
	\label{eq:Qliminf}
		Q(r) \leq \liminf_{m'} Q_{m'}(r),
	\end{equation}
	and there is a \( \lambda \in [0,1] \) such that \( Q(r) \to \lambda \) as \( r \to \infty \).
\end{lemma}

\begin{proof}
	As in \cite{Ambrosetti:2007db}, p.~252ff., and because the \( u_{m} \) are positive, the \( Q_{m} \) are all nondecreasing functions \( [0,\infty) \to [0,1] \), with \( \lim_{r \to \infty} Q_{m}(r) = 1 \). Therefore a Helly's selection argument shows that we can find a subsequence \( m' \subseteq m \) such that \( Q_{m'} \) converges to a function \( Q \) almost everywhere on \( [0,\infty) \): in particular, \(Q\) maps \( [0,\infty) \to [0,1] \), is nondecreasing, and, being nondecreasing, is continuous on the complement of a countable set, and hence can be adjusted to be continuous from the left, so
	\begin{equation*}
		Q(r) = \sup_{s<r} Q(s) = \sup_{s<r} \left( \liminf_{m'} Q_{m'}(s) \right) \leq \liminf_{m'} \left( \sup_{s<r} Q_{m'}(s) \right) \leq \liminf_{m'} Q_{m'}(r)
	\end{equation*}
	Being nondecreasing and bounded between \(0\) and \(1\), \(Q(r)\) also has a definite limit in \( [0,1] \) as \(r \to \infty\).\footnote{\cite{Korner:2004fk}, p.~9, for example.}
\end{proof}

\begin{proposition}
\label{thm:minxytightness}
	Let \( (u_{m}) \) be a sequence of functions in \( \cB \), with \( C[u_{m}] \) bounded above uniformly and \(\norm{u_{m}}_{2}=1\). Then there is a subsequence \( m' \subseteq m \) so that for any \( \varepsilon>0 \), there is \( R>0 \) such that
	\begin{equation*}
		\int_{\abs{x}>R} u_{m'}^{2} < \varepsilon
	\end{equation*}
\end{proposition}

\begin{proof}
	Following ideas from the dichotomy case of the Concentration--Compactness Principle,\footnote{\cite{Lions:1984rw,Lions:1984dq}} we shall prove that \( \lambda < 1 \) contradicts the uniform boundedness of \( C[u_{m}^{2}] \). Since \( Q_{m}(r) \to Q(r) \) almost everywhere and \( Q(r) \to \lambda \), for a given \( \varepsilon>0 \) we can choose a sequence \( R_{m} \to \infty \) such that
	\begin{equation*}
		Q_{m}(R_{m}) \leq \lambda + \varepsilon
	\end{equation*}
	for sufficiently large \(m\). Therefore
	\begin{equation*}
		1 - Q_{m}(R_{m}) \geq (1- \lambda) - \varepsilon,
	\end{equation*}
	and by the pigeonhole principle, one of
	\begin{align*}
		P_{m}^{+} &:= \int_{R_{m}}^{\infty} u_{m}^{2} \, dx &  P_{m}^{-} &:= \int_{-\infty}^{-R_{m}} u_{m}^{2} \, dx
	\end{align*}
	is larger than \( (1-\lambda-\varepsilon)/2 \). From the previous section, \( C[f] \) can be written in the form
	\begin{equation}
		C[f] = \int_{0}^{\infty} \left[ \left( \int_{r}^{\infty} f(x) \, dx \right)^{2} + \left( \int_{r}^{\infty} f(-x) \, dx \right)^{2} \right] \, dr,
	\end{equation}
	and the integrand is nonincreasing, we have
	\begin{align*}
		C[u_{m}^{2}] &\geq \int_{0}^{R_{m}} \left[ \left( \int_{r}^{\infty} u_{m}^{2} \, dx \right)^{2} + \left( \int_{-\infty}^{-r} u_{m}^{2} \, dx \right)^{2} \right] \, dr \\
		&\geq R_{m} \left( (P_{m}^{+})^{2} + (P_{m}^{-})^{2} \right) \\
		&\geq \frac{R_{m}}{4} \left( (1-\lambda) - \varepsilon \right)^{2}.
	\end{align*}
	This cannot be bounded above unless \( \lambda = 1-\varepsilon \), but since \( \varepsilon \) can be as close to \(0\) as we like, we must have \( \lambda = 1 \).
	
	Now, since \( Q(r) \to 1 \) as \( r \to \infty \), for any \( \varepsilon > 0 \) there is an \( R \) so that \( Q(R) \geq 1-\varepsilon/2 \), and the specification in equation \eqref{eq:Qliminf} implies that \( Q_{m}(R) \geq 1-\varepsilon \) for \( m \) sufficiently large. 
\end{proof}

This proof can in fact be generalised to more functions: see the Appendix.

\begin{theorem}
\label{thm:BsubsubL2}
	\( \cB \Subset L^{2} \).
\end{theorem}

This is a consequence of the following corollary of the Kolmogorov--Riesz compactness criterion:\footnote{\cite{Pego:1985fk}, quoted in \cite{HancheOlsen2010385}.}

\begin{proposition}
\label{thm:Pego-RK}
	Let \( \mathcal{F} \subset L^{2}(\reals^{d}) \) be such that \( \sup_{f \in \mathcal{F}} \norm{f}_{2} \leq M < \infty \). If
	\begin{equation*}
		\lim_{r \to \infty} \sup_{f \in \mathcal{F}} \int_{\abs{x}>r} \abs{f(x)}^{2} \, dx = 0 \quad \text{and} \quad \lim_{r \to \infty} \sup_{f \in \mathcal{F}} \int_{\abs{k}>r} \abs{\tilde{f}(k)}^{2} \, dk  = 0,
	\end{equation*}
	where \( \tilde{f} \) is the Fourier transform of \( f \), then \( \mathcal{F} \) is totally bounded in \( L^{2}(\reals^{d}) \).
\end{proposition}

\begin{proof}[Proof of Theorem]
	Let \( \cB_{M} = \{ u \in \cB : \norm{u}_{B} \leq M < \infty \} \), and consider a sequence \( (u_{m}) \subset \cB_{M} \). We show that this sequence satisfies the conditions of the Proposition.
	
	Since \( \cB \subset H^{1} \), we have \( \int k^{2} (\tilde{u}_{m}(k))^{2} \, dk \leq M < \infty \), and it follows that we must have
	\begin{equation*}
		\lim_{r \to \infty} \sup_{m} \int_{\abs{k}>r} \abs{\tilde{u}_{m}(k)}^{2} \, dk = 0,
	\end{equation*}
	which is the same as the second condition in the Proposition.
	
	For the first condition, we can proceed by contradiction: suppose that
	\begin{equation}
		\lim_{r \to \infty} \sup_{u \in \cB_{M}} \int_{\abs{x}>r} \abs{u(x)}^{2} \, dx = \alpha > 0.
	\end{equation}
	Then there are sequences \( (u_{m}) \subset \cB_{M} \), \( (R_{m}) \subset \reals_{>0} \) so that
	\begin{equation}
		\int_{\abs{x}>R_{m}} \abs{u_{m}(x)}^{2}  \, dx \geq \tfrac{1}{2}\alpha,
	\end{equation}
	and since by \autoref{thm:minxytightness} this cannot occur, \( \cB_{M} \) must satisfy the requirements of \autoref{thm:Pego-RK}.
\end{proof}

\subsection{Existence of a Minimiser}

Since \( E \) is bounded below on \( \cB \), we can define
\begin{equation}
	e_{0} = \inf_{v \in \cB_{1}} E[v].
\end{equation}

We shall prove

\begin{theorem}
\label{thm:z=1exists}
	There is \( u \in \cB_{1} \) with \( E[u] = e_{0} \).
\end{theorem}

\begin{lemma}
\label{thm:E1WLSC}
	\( E \) is weakly lower-semicontinuous on \( \cB \).
\end{lemma}
\begin{proof}
	The proof is the same as in the \( z>1 \) case, \autoref{thm:z>1wlsc}, replacing $g$ with $G$.
\end{proof}

\begin{proof}[Proof of Theorem]
	By the definition of \( e_{0} \), there is a sequence \( u_{m} \subset \cB_{1} \) such that \( E[u_{m}] \to e_{0} \) as \( m \to \infty \). Clearly we can pick these \( u_{m} \) so that \( E[u_{m}] \) is uniformly bounded above in \( \cB \). Since \( \cB \) is reflexive, the Banach--Alaoglu theorem implies that there is a subsequence \( m' \subset m \) and a function \( u \in \cB \) such that \( u_{m'} \tow u \).
	
	The compactness theorem, \autoref{thm:BsubsubL2}, implies that \( \cB \Subset L^{2} \), and hence \( \norm{u}_{2} = 1 \). (Since \(L^{2}\) is a Hilbert space, this also gives us \( u_{m'} \to u \) strongly in \(L^{2}\).)
	
	By \autoref{thm:E1WLSC}, \( E \) is WLSC on \( \cB \). Hence
	\begin{equation}
		e_{0} = \lim_{m' \to \infty} E[u_{m'}] \geq E[u],
	\end{equation}
	but \( u \in \cB_{1} \), so \( E[u] \geq e_{0} \) by definition, and hence we must have \( E[u] = e_{0} \).
\end{proof}

\section{Symmetric Non-increasing: A (Strict) Rearrangement inequality for \texorpdfstring{\( C[f] \)}{C[f]}}

\begin{lemma}
	Let $g: \reals^{n} \to \reals $ be a spherically symmetric nondecreasing function (that is, if \( t \geq 0 \) and \( n \) is a unit vector, \( g(t n ) \) is a nondecreasing function of \(t\) and has the same value for any choice of \(n\)). Then, for any \( 0<a\leq \infty\), if \( f\geq 0 \) is rearrangeable (that is, measurable and with \( \mu\{ s: \abs{f}>s\} < \infty \) for \( s > 0 \)), we have
	\begin{equation}
		\int  f(x) g(x)\, dx \geq \int  f^{*}(x) g(x)\, dx,
	\end{equation}
	where \(f^{*}\) is the symmetric-decreasing rearrangement of \(f\).\footnote{See e.g. \cite{lieb2001analysis}, Ch.~3.} Further, if \( g \) is strictly increasing and \( f \) is not almost everywhere equal to \( f^{*} \), then the inequality is strict.
\end{lemma}

One can compare this to the standard Hardy--Littlewood inequality,\footnote{\cite{HLP:1952}, Theorem 378 (p.278). The reason for the attribution to only Hardy and Littlewood is unclear, given that no trace of it is found prior to the proof given there.} in which both functions are rearranged, or the Riesz-type inequality discussed by Choquard and Stubbe \cite{Choquard:2007fk} in their proof, but this is somehow more intuitive and basic.

\begin{proof}
	We may assume that \(g\) is bounded below: if not, we consider the truncations \( g_{n}(x) = \max{\{ g(x),-n \}} \), and then apply the monotone convergence theorem.
	
	For a \(g\) that is bounded below, it suffices to check for \( f \) the characteristic function of a measurable set, \( f=\chi_{A} \). Therefore we need to prove that
	\begin{equation}
		\int_{A} g \geq \int_{A^{*}} g,
	\end{equation}
	where \(A^{*}\), the spherically symmetric rearrangement of \(A\), is the ball of volume \( \mu(A) \) centred at \(0\), with radius $m \geq 0$. It is apparent that the points of \( A \cap A^{*} \) appear on both sides of the inequality, so they are irrelevant. Therefore we have to show
	\begin{equation*}
		\int_{A \setminus A^{*}} g \geq \int_{A^{*} \setminus A} g.
	\end{equation*}
	We have the well-known bounds
	\begin{equation}
	\label{eq:trivialbounds}
		\mu(B) \inf_{B}{g} \leq \int_{B} g \leq \mu(B) \sup_{B}{g};
	\end{equation}
	since \( \mu(A) = \mu(A^{*}) \) and \( g \) is symmetric nondecreasing, we can apply the left inequality to \( A \setminus A^{*} \), the right to \( A^{*} \setminus A \) to obtain
	\begin{align*}
		\int_{A \setminus A^{*}} g &\geq \mu(A \setminus A^{*}) g(m) \\
		&\geq \int_{A^{*} \setminus A} g,
	\end{align*}
	as required, and we then extend using linearity and monotone convergence.
	
	To prove the second part, notice that equality in \eqref{eq:trivialbounds} requires that 
	\begin{equation*}
		\mu( \{x: \essinf_{B} g \neq g \neq \esssup_{B} g \} ) = 0,
	\end{equation*}
	so either \( \mu(B)=0 \) or \( g \) is essentially constant on \( B \), neither of which can happen if \( g \) is strictly increasing and  \( f \) is not a.e. equal to \( f^{*} \) (a non-negligible collection of its level sets will have \( A \setminus A^{*} \) non-null\footnote{See Lieb and Loss, \cite{lieb2001analysis} Theorem 3.4 (p.82), for the details of the argument.}).
\end{proof}

We now need an alternative to this to deal with the double integral in \(C\).

\begin{proposition}
	Let \(f,g \in L^{1}(\reals) \) be locally integrable, and suppose \(f,g \geq 0\). Let \(h(x,y)\) be even in \(x\), even in \(y\) and, for all \(y\), a nondecreasing function of \(x\) for \(x>0\), and similarly in \( y \) for all \( x \). Then
	\begin{equation}
		\iint f(x) g(y) h(x,y) \, dx \, dy \geq \iint f^{*}(x) g^{*}(y) h(x,y) \, dx \, dy.
	\end{equation}
	If ``nondecreasing'' is replaced by ``strictly increasing'', \(f=g\), and \(f\) is not almost everywhere equal to \(f^{*}\), the inequality is strict.
\end{proposition}

\begin{proof}
	Again, we shall prove the proposition first for characteristic functions of measurable sets; the rest of the argument follows in the usual way. Let \(f=\chi_{A}\), \( g=\chi_{B} \). This time we have to split \( A \times B \) into \( 4 \) regions:
	\begin{align}
		(A \cap A^{*}) \times (B \cap B^{*}), \quad (A \setminus A^{*}) \times (B \cap B^{*}), \quad (A \cap A^{*}) \times (B \setminus B^{*}), \quad (A \setminus A^{*}) \times (B \setminus B^{*}).
	\end{align}
	The integral over the first of these is unchanged on replacing \(A\) and \(B\) by their respective rearrangements \(A^{*}\) and \(B^{*}\). Importantly, since the middle two sets are only translated parallel to one axis in the transition \( A \to A^{*} \), \( B \to B^{*} \), and \( h(x,y) \) is nondecreasing on lines of constant \(x\) and constant \(y\),  the integrals over the middle two cannot increase. Finally, both of the coordinates of the last set must decrease in absolute value, and it follows by the same argument as in the proof of the previous theorem that
	\begin{align*}
		\int_{A \setminus A^{*}} \int_{B \setminus B^{*}} h(x,y) \, dx \, dy &\geqslant \mu(A \setminus A^{*})\mu(B \setminus B^{*}) h(\mu(A)/2,\mu(B)/2) \\
		&\geqslant \int_{A^{*} \setminus A} \int_{B^{*} \setminus B} h(x,y) \, dx \, dy,
	\end{align*}
	and hence the inequality follows by the same reasoning as before.
	
	The proof of the strictness is also similar, relying on the strictness of the last inequality when \( h(x,y) \) is strictly increasing in both arguments and \( \mu(A \setminus A^{*}) > 0 \).
\end{proof}

\begin{corollary}[Extension to \(n\) functions]
	Suppose \( (f_{i})_{i=1}^{n} \subseteq L^{1}(\reals) \) satisfy \(f_{i}\geq 0\) and \(h(x_{1}, \dotsc , x_{n})\) is real-valued function, even in each argument, which is nondecreasing as a function of \(x_{i}\)for all the other entries fixed and \(x_{i}>0\). Then
	\begin{equation}
		\idotsint \left( \prod_{i=1}^{n} f(x_{i}) \right) h(x_{1},\dotsc,x_{n}) \, dx_{1} \dotsm \, dx_{n} \geq \idotsint \left( \prod_{i=1}^{n} f^{*}(x_{i}) \right) h(x_{1},\dotsc,x_{n}) \, dx_{1} \dotsm \, dx_{n}.
	\end{equation}
	If ``nondecreasing'' is replaced by ``strictly increasing'', and \(f_{i}=f\), if \(f\) is not almost everywhere equal to \(f^{*}\) the inequality is strict.
\end{corollary}

The proof of this is obviously carried out analogously to the previous theorem, but with \( 2^{n} \) regions instead of \( 4 \).

However, the result we actually need is:
\begin{proposition}
	Suppose \(f\geq 0\) and \(h(x,y)\) is a symmetric real-valued function which for all \(y\), is a nondecreasing function of \(x\) for \(x>0\), and a nondecreasing function of \( -x \) for \(x<0\) (i.e. \(h(-x,y)\) is nondecreasing for \( x >0 \)). Then
	\begin{equation}
		\iint f(x) f(y) h(x,y) \, dx \, dy \geq \iint f^{*}(x) f^{*}(y) h(x,y) \, dx \, dy,
	\end{equation}
	If \(h\) is strictly increasing on almost all diagonals \(y=kx\) for \( k>0 \), and \(f\) is not almost everywhere equal to \(f^{*}\), then the inequality is strict.
\end{proposition}

\begin{proof}
	The first part of the proof is identical to that of the previous proposition. Now, suppose the assumptions of the second part are satisfied. Once again we consider \( f=\chi_{A} \). Set \( \mu(A)=2m \). Since \(f\) is not equivalent to \(f^{*}\), there is a measurable subset \( B \subseteq A \setminus A^{*} \) and \( \varepsilon>0 \) so that \( \mu(B)>0 \) and \( B \) is outside the interval \( [-m-\varepsilon,m+\varepsilon] \). We only need to consider \( x,y>0 \), for \( x,y<0 \) the argument is near-identical, and one of these will always contain a set of positive measure. (The case \(xy<0\) is not required, since we know that these terms cannot increase by the previous results.) There are now three cases: \( x>y \), \( x=y \), and \( x<y \). For \( x=y \), \( h(x,x)>h(m+\varepsilon,m+\varepsilon)>h(m,m) \), and there is no more to say. Suppose now \( x>y \) (the other case shall follow using the symmetry of \(h\)). Then for each \( (x,y) \in B \times B \), there is \( k>1 \) so that \( x>ky>m+\varepsilon \), so
	\begin{equation}
	h(x,y) \geq h(ky,y) > h\left(k(m+\varepsilon),m+\varepsilon\right) \geq h(m+\varepsilon,m+\varepsilon) > h(m,m).
	\end{equation}
	It follows that
	\begin{equation*}
		\inf_{x,y\in B} h(x,y) \geq h(m+\varepsilon,m+\varepsilon) > h(m,m),
	\end{equation*}
	so for this \( B \) the inequality is strict. The same argument as before shows that for any \( f \) not almost-everywhere equal to \( f^{*} \) we have enough such sets to make a strict inequality for the whole function.
\end{proof}

Now, if \( z \leq 1 \), our strict rearrangement inequality above applies to the potential energy term \( B \), so we have
\begin{equation*}
	B[u^{2}] > B[(u^{*})^{2}]
\end{equation*}
for any \(u\) that is not symmetric nonincreasing, and combining this with the P\'olya--Szeg\H{o} inequality for the kinetic energy,\footnote{originally proven in \cite{10.2307/2371912}} we have
\begin{equation*}
	E[u] > E[u^{*}]
\end{equation*}

\begin{corollary}
	Any minimiser of \(E\) in the appropriate space\footnote{I.e. \( \cX_{1} \) for $z> 1$, \( \cB_{1} \) for \(z=1\)} is symmetric nonincreasing.
\end{corollary}

\section{Existence for a More General Background}

We can use much the same technique as in the previous sections to prove the existence of a bound state for a more general potential:

\begin{theorem}
	Let \( \rho(x)<0 \), \( \int_{\reals} \abs{x}\abs{\rho(x)} \, dx < \infty \), and \( \int \rho = -z \leq -1 \). Then,
	\begin{enumerate}
	\item
	if \( z > 1 \), there is \( u \in \cX \) such that \( E[u] = \inf_{v \in \cX} E[v] \),
	\item
	if \(z=1\), there is \( u \in \cB_{1} \) such that \( E[u] = \inf_{v \in \cB_{1}} E[v] \),
	\end{enumerate}
	where \(E\) is given by \eqref{eq:energy}.
\end{theorem}

The proof is of the same form as that of the previous theorem, but we also use the following:

In both cases, we have as before that \( \int_{\reals} u^{2} =1 \), so
\begin{align*}
	B_{\rho}[u] &:= -\frac{1}{2}\int_{\reals} \int_{\reals} \abs{x-y}(u(x)^{2}+\rho(x))(u(y)^{2}+\rho(y)) \, dx \, dy \\
	&= \int_{\reals} \int_{\reals} \left( V_{\rho}(x)+V_{\rho}(y)-\frac{1}{2}\abs{x-y} \right) u(x)^{2} u(y)^{2} \, dx \, dy + \text{const.},
\end{align*}
as in the previous argument; the constant is finite since \( \int \rho V_{\rho} < 2 \left(\int \rho \right) \left(\int \abs{x} \rho(x) \, dx\right) < \infty \). Hence, we can take \( G(x,y) = V_{\rho}(x)+V_{\rho}(y)-\frac{1}{2}\abs{x-y} \), and proceed as before. Jensen's inequality also gives us
\begin{align*}
	z^{-1}V_{\rho}(x) &= \frac{1}{2}\int \abs{x-y} (-z^{-1}\rho(y))\, dy \\
	&\geq \frac{1}{2} \abs{ x\int (-z^{-1}\rho(y)) \, dy - \int y(-z^{-1}\rho(y)) \, dy } \\
	&= \frac{1}{2}\abs{x-z^{-1}\int y (-\rho(y)) \, dy}.
\end{align*}
Therefore, if we set \( P= -z^{-1}\int w \rho(w) \, dw \) and change variables to \( X= x-P \), \( Y = y-P \), we have
\begin{align*}
	\int_{\reals} \int_{\reals}& \left( V_{\rho}(x)+V_{\rho}(y)-\frac{1}{2}\abs{x-y} \right) u(x)^{2} u(y)^{2} \, dx \, dy \\
	&= \int_{\reals} \int_{\reals} \left( V_{\rho}(X+P)+V_{\rho}(Y+P)-\frac{1}{2}\abs{X-Y} \right) u(X-P)^{2} u(Y-P)^{2} \, dX \, dY \\
	&\geq \int_{\reals} \int_{\reals} \left( \frac{z}{2}(\abs{X}+\abs{Y})-\frac{1}{2}\abs{X-Y} \right) u(X-P)^{2} u(Y-P)^{2} \, dX \, dY,
\end{align*}
which is the same as the previous functional we considered. Hence this bound allows us to use the same proofs for coercivity, and WLSC as before. It remains to check that \( B_{\rho} \) is finite on the appropriate space in each case.

In both cases, we have the useful bound
\begin{equation}
\label{eq:VrhoUpperBound}
	V_{\rho}(x) \leq \frac{1}{2} \int_{\reals} (\abs{x}+\abs{y}) (-\rho(y)) \, dy = \frac{z}{2} \abs{x} + \frac{1}{2}\int_{\reals} \abs{y} \abs{\rho(y)} \, dy,
\end{equation}
so in fact
\begin{equation}
	\abs{V_{\rho}(x)-\frac{z}{2} x} \leq \frac{1}{2}\int_{\reals} \abs{y} \abs{\rho(y)} \, dy.
\end{equation}
In the \(z>1\) case, we can use this to show the finiteness of the \( \int_{\reals} u^{2} V_{\rho} \) terms for \( u \in \cX \) in the \(z>1\); the other is finite as in the \(\rho=-z\delta \) case.

On the other hand, in the \(z=1\) case, we exploit the inner product structure provided by \(B\): if we set
\begin{equation}
	\langle f,g \rangle = \int_{\reals} \int_{\reals} -\abs{x-y} f(x) g(y) \, dx \, dy,
\end{equation}
then we have

\begin{proposition}
	Let \( K = \{ f \in C_{c}^{\infty}(\reals) : \int_{\reals} f = 0 \} \). Then \( \langle \cdot , \cdot \rangle \) is an inner product on \(K\).
\end{proposition}

This appears rather surprising, since the kernel is negative; however, it is also proportional to the inverse Laplacian. Intuitively, by Plancherel's theorem
\begin{equation*}
	\langle f,f \rangle = \int_{\reals} k^{-2} \abs{\tilde{f}(k)}^{2} \, dk,
\end{equation*}
but the question of integrability at zero is a substantial difficulty. We instead appeal to the solution to Poisson's equation using the inverse Laplacian.

\begin{proof}
	Let \( f \in K \). Then
	\begin{equation*}
		u(x) = -\frac{1}{2}\int_{\reals}\abs{x-y} f(y) \, dy
	\end{equation*}
	is well-defined, convergent, and solves Poisson's equation \( -u'' = f \). Multiplying both sides of this by \(2u\) and integrating, we have
	\begin{equation*}
		-\int_{-a}^{b} 2uu'' = -\int_{-a}^{b}\int_{\reals}\abs{x-y} f(x) f(y) \, dy \, dx = \langle f,f \rangle
	\end{equation*}
	if \(a\) and \(b\) are large enough, since \(f\) has compact support. Integrating by parts, the left-hand side becomes
	\begin{equation*}
	-\int_{-a}^{b} 2uu'' = -2u(b)u'(b)+2u(-a)u'(-a) + 2\int_{-a}^{b} u'^{2}.
	\end{equation*}
	The second term is clearly positive, and we need to verify that the boundary terms do not contribute. We know
	\begin{equation*}
		2u'(x) = - \int_{y<x} f(y) \, dy + \int_{y>x} f(y) \, dy,
	\end{equation*}
	and both of these are zero for sufficiently large (or sufficiently negative) \(x\) since \( \int_{\reals} f = 0 \) and \(f\) has compact support. It follows that \( \langle f,f \rangle = 2\int_{-a}^{b} u'^{2} \) is nonnegative, and it is also clear that it is zero if and only if \(f \equiv 0\), as required.
\end{proof}

We now complete \(K\) to the space
\begin{equation}
	Y = \left\{ f \in L^{1}(\reals) : \int_{\reals} f = 0, \langle f , f \rangle < \infty \right\}.
\end{equation}
using the norm \( (\norm{f}_{1}^{2}+\langle f,f \rangle)^{1/2} \); this is a Hilbert space, and \( \langle \cdot , \cdot \rangle \) remains an inner product on \( Y \), by a standard result in metric spaces.\footnote{See, e.g., \cite{Korner:2004fk}, p.~340, Lemma 14.11.}

Hence, it satisfies the triangle inequality, in particular, if \( \int_{\reals} \abs{x}\abs{\rho(x)} \, dx < \infty \) and \( \int \rho = -1 \), we have
\begin{equation}
	\abs{\langle u^{2} +\rho , u^{2}+\rho \rangle^{1/2} - \langle u^{2}-\delta , u^{2}-\delta \rangle^{1/2}} \leq \langle \delta + \rho, \delta+\rho \rangle^{1/2};
\end{equation}
the latter may be calculated to be the square root of \( \langle \rho,\rho \rangle + 2\int \abs{x} \rho(x) \, dx  \), which we know is finite by \eqref{eq:VrhoUpperBound} and the conditions in the theorem. Hence the Coulomb term \( \langle u^{2} +\rho , u^{2}+\rho \rangle \) is finite and coercive precisely when \( C[u^{2}] \) is; this gives us the theorem in the \(z=1\) case.

\section{Conclusion}

We have shown the existence of a stationary solution to the Maxwell--Schr\"odinger equations in one dimension, for a general fixed background charge distribution satisfying the natural condition \( \int_{\reals} \abs{x}\abs{\rho(x)} \, dx < \infty \). The case when there is more fixed than moveable charge is simple, and can be handled using only conditions similar to that mentioned in the previous sentence, but the \(z=1\) case required a more subtle argument, using the quartic Banach space \( \cB \) discussed in \S~\ref{sec:Bspaces}. The general case has been shown to follow simply from the case of the singular case \(\rho=\delta\), so the fundamental solution of the Coulomb equation continues to be useful even in this nonlinear problem.

There remains the question of whether the condition \( \int_{\reals} \abs{x}\abs{\rho(x)} \, dx < \infty \) is the weakest possible. One may consider instead asking about potentials such that \( \int_{\reals} \int_{\reals} \abs{x-y} \rho(x) \rho(y) \, dx \, dy < \infty \) but certainly the theory becomes considerably more complicated in this case: one may no longer define the potential \( V_{\rho} \), for example. An even more difficult case would be that in which both \( B[\rho]=\infty \) and, \(B[u^{2}]=\infty \), but \( B[u^{2}+\rho]<\infty \).

The author is also hopeful that the discussion of similar equations in two dimensions may benefit from these results, the large-scale behaviour being of a similar character (in that potentials for non-neutral charge distributions diverge at \(\infty\)), although the singularities involved make using the fundamental solution as an initial step a rather less attractive proposition. One might also ask about the existence of excited states, which the author is currently investigating using similar methods to \cite{Choquard:2008oq}.

\appendix

\section{Generalisation of Rellich's criterion}

\newcommand{\vx}{\underline{X}}
\newcommand{\vr}{\underline{R}}

\subsection{One Dimension}

We begin by considering \( u: \reals \to \complexes \). Write \( \vx = (x_{1},\dotsc,x_{n}) \), \( x_{i} \in \reals \).\footnote{We shall follow the notation of Lieb and Seiringer \cite{Lieb:2010nx}, although our vectors are somewhat less bold.} Write also \( \vx>0 \) for \( \bigwedge_{i=1}^{n} (x_{i}>0) \), and \( \abs{\vx}=(\abs{x_{1}},\dotsc,\abs{x_{n}}) \). Throughout we consider functionals of the form
\begin{equation}
	I_{G}[u] = \idotsint G(\vx) \prod_{i=1}^{n} \abs{u(x_{i})}^{2} \, d\vx;
\end{equation}
therefore it will be convenient to use the shorthand
\begin{equation}
	U(\vx) = \prod_{i=1}^{n} u(x_{i}),
\end{equation}
so that in particular \( \abs{U(\vx)}^{2} = \abs{u(x_{1})}^{2} \dotsm \abs{u(x_{n})}^{2} \).

We can see that \autoref{thm:minxytightness} can be generalised to any nonnegative function \( G(\vx) \) which satisfies one of the following three properties:
\begin{enumerate}
	\item
	\( G(\vx) \to \infty \) as \( \norm{\vx} \to \infty \) (i.e. \( \forall N \exists r_{N} \) such that \( G(\vx) > N \) when \(\norm{\vx}>r_{N}\)), where \( \norm{\vx} \) is a norm on \( \reals^{n} \).\footnote{Of course, any norm on \( \reals^{n} \) can be used since they are all equivalent.}
	\item
	\( G(x,\dotsc,x) \to \infty \) as \( \abs{x} \to \infty \) and \( G(\vx) \) and \( G(-\vx) \) are increasing for \( \vx>0 \) (i.e. if \( \vx>\vx' \), \( G(\vx) > G(\vx')  \)).
	\item
	There is an increasing function \( \vx(t) \) such that \( G( \pm \vx(t)) \to \infty \) as \( t \to \infty \) and \( G(\vx) \) and \( G(-\vx) \) are increasing for \( \vx>0 \).
\end{enumerate}

In particular,
\begin{lemma}
	Let \( G: \reals^{n} \to \reals^{+} \) satisfy one of the three conditions, and let \( u \in S := \{ u \in L^{2}(\reals): \norm{u}_{2}<1, I_{G}[u]<1 \} \). Then \( \int_{\abs{x}>r} \abs{u(x)}^{2} \, dx \to 0 \) as \( r \to \infty \), uniformly in \(S\).
\end{lemma}

The proofs of these are entirely analogous to those done before: for the first one, we have
\begin{align*}
	I_{G}[u] &\geq \idotsint_{\norm{\vx}>r_{N}} G(\vx) \abs{U(\vx)}^{2} d\vx \\
	&\geq N \idotsint_{\norm{\vx}>r_{N}} \prod_{i} \abs{u(x_{i})}^{2} d\vx \\
	&\geq N \prod_{i} \int_{\abs{x_{i}}>r_{N}} \abs{u(x_{i})}^{2} dx_{i} \\
	&= N \left( \int_{\abs{x}>r_{N}} \abs{u(x)}^{2} \, dx \right)^{n},
\end{align*}
and so the integral in the last line tends to zero uniformly in \( u \) as before.

The second one is a special case of the third with \( \vx(t) = (t,\dotsc,t) \), and the third can be done using much the same idea: let \( r_{N} \) be such that \( G(\pm\vx(t))>N \) for \( t>r_{N} \).
\begin{align*}
	I_{G}[u] &\geq \idotsint_{\vx'>\vx(r_{N})} G(\vx') \abs{U(\vx')}^{2} d\vx' + (-) \\
	&\geq G(\vx(r_{N})) \idotsint_{\vx'>\vx(r_{N})} \prod_{i} \abs{u(x'_{i})}^{2} d\vx' + (-) \\
	&\geq N \idotsint_{\vx'>\vx(r_{N})} \prod_{i} \abs{u(x'_{i})}^{2} d\vx' + (-) \\
	&\geq N \prod_{i} \int_{x'_{i}>x_{i}(r_{N})} \abs{u(x'_{i})}^{2} dx'_{i} + (-) \\
	&\geq N \left( \int_{x'>x_{N}} \abs{u(x')}^{2} \, dx' \right)^{n} + (-),
\end{align*}
where \( x_{N}=\max_{i}{\{ x_{i}(r_{N}) \}} \), and \((-)\) denotes the same term with \( \vx,\vx',r_{N},> \) replaced by their negatives.

There is another use of this result: the Rellich criterion also specifies control over the Fourier space; as above this is explained in the papers \cite{Pego:1985fk,HancheOlsen2010385}, by converting uniform convergence of \( f(\cdot-y) \to f \) into decay of the Fourier transform \(\tilde{f}\) of \(f\). We can then use the above lemma to provide the following partial extension of the Rellich criterion:

\begin{theorem}
\label{thm:Rellich1DGen}
	Let \( F: \reals^{n} \to \reals^{+} \), satisfy one of the three above conditions, and let \( G: \reals^{m} \to \reals^{+} \), satisfy one of the three above conditions. Then the set
	\begin{equation}
		\left\{ u \in L^{2}(\reals) : \norm{u}_{2}\leq 1, I_{F}[u] \leq 1, I_{G}[\tilde{u}]\leq 1 \right\}
	\end{equation}
	is compactly embedded in \( L^{2}(\reals) \).
\end{theorem}

\subsection{\texorpdfstring{\(d\)}{d} Dimensions}
\renewcommand{\vr}{\mathbf{r}}

We provide only a very simple generalisation, which shall prove sufficient for our purposes.

\begin{lemma}
	Let \( \vx = (x_{1},x_{2},\dotsc,x_{n}) \in (\reals^{d})^{n} \), define \( \vr = \abs{\vx} \), suppose \( G(\vx) = g(\abs{\vx}) \) is componentwise-radial, and let \( g: (\reals^{+})^{n} \to \reals^{+} \) satisfy one of the three conditions from the previous section. Let \( u \in S := \{ u \in L^{2}(\reals^{d}): \norm{u}_{2}<1, I_{G}[u] < 1 \} \). Then \( \int_{\abs{x}>r} \abs{u(x)}^{2} \to 0 \) as \( r \to \infty \), uniformly in \(S\).
\end{lemma}
(We may extend \( g \) to negative numbers by imposing that it is even, although this is unimportant in the proof)

\begin{proof}
	We have, by Tonelli's theorem
	\begin{align*}
		I_{G}[u] &= \int_{\reals^{d}} \dotsi \int_{\reals^{d}} G(\vx) \abs{U(\vx)}^{2} \, d\vx \\
		&= \int_{\abs{\vx}>0} g(\abs{\vx}) \prod_{i} r_{i}^{d-1} \int_{\norm{n_{i}}_{2}=1} \abs{u(r_{i}n_{i})}^{2} dn_{i} d\vr \\
		&= \int_{\abs{\vx}>0} g(\abs{\vx}) \prod_{i} \abs{\bar{u}(r_{i})}^{2} d\vr,
	\end{align*}
	where we define
	\begin{equation}
		\abs{\bar{u}(r)}^{2} = r^{d-1}\int_{\norm{n}_{2}=1} \abs{u(rn)}^{2} \, dn,
	\end{equation}
	proportional to the spherical mean of \(\abs{u}^{2}\). But then \(I_{G}[u]\) is in the same form as the one-dimensional case, with \( \bar{u} \) serving as the \(L^{2}\) function. Hence the previous proof applies, and shows that
	\begin{equation*}
		\lim_{R \to \infty} \int_{r>R} \abs{\bar{u}(r)}^{2} \, dr = 0.
	\end{equation*}
	But
	\begin{equation*}
		\int_{r>R} \abs{\bar{u}(r)}^{2} \, dr = \int_{\abs{x}>R} \abs{u(x)}^{2} \, dx,
	\end{equation*}
	and so we obtain the desired result.
\end{proof}

We then can obtain the obvious extension of \autoref{thm:Rellich1DGen}, where we replace \( L^{2}(\reals) \) by \( L^{2}(\reals^{d}) \).

\begin{remark}
	Although we used \( L^{2} \) here, this is obviously not essential: one can extend the results to \( L^{p} \) and \( L^{q} \), where \( p^{-1}+q^{-1}=1 \); the proof is effectively identical, although the underlying theory is slightly different.
\end{remark}

\section*{Acknowledgements}

The author would like to thank David Stuart for many helpful discussions over the course of this paper's gestation.

\bibliographystyle{spmpsci}
\bibliography{1DNC}

\end{document}